\documentclass[12pt]{amsart}
\usepackage[french, british]{babel}
\usepackage{color}
\usepackage[centertags]{amsmath}
\usepackage[margin=1in]{geometry}
\usepackage{amsfonts}
\usepackage{amsthm}
\usepackage{newlfont}
\usepackage{verbatim}
\usepackage{amscd}
\usepackage{amsgen}
\usepackage{amssymb}

\usepackage{amssymb,amsmath}
\usepackage{amscd}
\usepackage{enumerate}
\usepackage{color}
\usepackage[all]{xy}
\usepackage{graphicx}

\newlength{\defbaselineskip} 
\newcounter{wsk}
\newcounter{wskk}
\theoremstyle{plain}
\newtheorem{thm}{Theorem}[section]
\newtheorem{cor}[thm]{Corollary}
\newtheorem{con}[thm]{Conjecture}
\newtheorem{df}[thm]{Definition}
\newtheorem{lema}[thm]{Lemma}
\newtheorem{obs}[thm]{Proposition}
\newtheorem{exm}[thm]{Example}
\newtheorem{question}[thm]{Question}

\newtheorem{rem}[thm]{Remark}
\newtheorem{pr}{Algorithm}
\theoremstyle{definition} 
\theoremstyle{definition} \newtheorem{ex}{Example}[section] %

\numberwithin{equation}{section}

\def\n{\mathbb{N}}

\def\N{\mathbb{N}}

\def\hook{\llcorner}

\def\cos{\downarrow}
\def\ob{\begin{obs}}
\def\kob{\end{obs}}
\def\dow{\begin{proof}}
\def\kdow{\end{proof}}

\def\tw{\begin{thm}}
\def\ktw{\end{thm}}
\def\hip{\begin{con}}
\def\khip{\end{con}}
\def\lem{\begin{lema}}
\def\klem{\end{lema}}
\def\ex{\begin{exm}}
\def\prog{\begin{pr}}
\def\kprog{\end{pr}}
\def\wn{\begin{cor}}
\def\kwn{\end{cor}}
\def\uwa{\begin{rem}}
\def\kuwa{\end{rem}}
\def\kex{\end{exm}}
\def\dfi{\begin{df}}
\def\kdfi{\end{df}}
\def\TT{{T\!T}}
\def\HF{HF}
\def\HFT{H\mathfrak{T}}
\def\TTT{T\mathfrak{T}}
\definecolor{zielony}{rgb}{0.2, 0.5, 0}
\definecolor{czerwony}{rgb}{0.9, 0.2, 0.1}
\definecolor{brazowy}{rgb}{0.5, 0.1, 0.0}
\definecolor{niebieski}{rgb}{0.3, 0.1, 0.9}
\newcommand{\red}[1]{\color{czerwony} #1 \color{black}}
\newcommand{\green}[1]{\color{zielony} #1 \color{black}}

\newcommand{\brown}[1]{\color{brazowy} #1 \color{black}}
\newcommand{\smalltextgreen}[1]{\green{\text{\scriptsize#1}}}

\newenvironment{redd}{\red}{}
\newcommand{\bred}{\begin{redd}}
\newcommand{\ered}{\end{redd}}

\newcommand{\textleaf}[1]{\smalltextgreen{leaf#1}}
\newcommand{\node}{{\text{\scriptsize node}}}
\newcommand{\textroot}{\brown{\text{\scriptsize root}}}
\xyoption{line}
\def\Tr{\mathfrak{T}}
\def\V{\Vertices}
\def\Vertices{\mathfrak{V}}
\def\NNN{\Nodes}
\def\Nodes{\mathfrak{N}}
\def\LL{\Leaves}
\def\Leaves{\mathfrak{L}}

\newcommand{\set}[1]{\left\{#1\right\}}
\newcommand{\fromto}[2]{#1, \dotsc, #2}
\newcommand{\setfromto}[2]{\set{\fromto{#1}{#2}}}

\def\Root{\mathfrak{r}}
\newcommand{\tree}{\Tr}
\DeclareMathOperator{\TNS}{TNS}
\newcommand{\TTformat}{\TT{} format}
\newcommand{\DOAD}[1]{\ensuremath{\operatorname{DOAD}(#1)}}

\begin{document}
\title{The Hackbusch conjecture on tensor formats}
 \author[W.~Buczy\'nska]{Weronika Buczy\'nska}
 \thanks{W.~Buczy\'nska is supported by Polish National Science Center, project 2013/08/A/ST1/00804}
 \address{Weronika Buczy\'nska\\
 Institute of Mathematics of the
 Polish Academy of Sciences\\
 ul. \'Sniadeckich 8\\
 00-656 Warszawa, Poland}
 \email{wkrych@mimuw.edu.pl}

 \author[J.~Buczy\'nski]{Jaros\l{}aw Buczy\'nski}
 \thanks{J.~Buczy\'nski is supported by Polish National Science Center, project 2013/11/D/ST1/02580,
          and by a scholarship of Polish Ministry of Science."}
 \address{Jaros\l{}aw Buczy\'nski\\
Faculty of Mathematics, Computer Science and Mechanics\\
University of Warsaw\\
ul. Banacha 2\\
02-097 Warszawa\\
Poland\\
and
Institute of Mathematics of the
  Polish Academy of Sciences\\
  ul. \'Sniadeckich 8\\
  00-656 Warszawa, Poland
 }
 \email{jabu@mimuw.edu.pl}

\author[M.~Micha{\l}ek]{Mateusz Micha{\l}ek}
 \thanks{M.~Micha{\l}ek is supported by a grant Iuventus Plus 0301/IP3/2015/73 of the Polish Ministry of Science}
 \address{Mateusz Micha{\l}ek\\
 Simons Institute for the
Theory of Computing\\ 121 Calvin Lab 2190\\ UC Berkeley, CA 94720, USA\\
 and
 Institute of Mathematics of the
 Polish Academy of Sciences\\
 ul. \'Sniadeckich 8\\
 00-656 Warszawa, Poland}
 \email{wajcha2@poczta.onet.pl}
\keywords{Hackbusch conjecture, tensor formats, hierarchical format, TT formats, complexity of tensors}
\subjclass[2010]{15A69 (primary); 46B28, 65D15 (secondary)}




\selectlanguage{british}
\begin{abstract}
We prove a conjecture of W.~Hackbusch about tensor network states related to a perfect binary tree and train track tree.
Tensor network states are used to present seemingly complicated tensors in a relatively simple and efficient manner.
Each such presentation is described by a binary tree and a collection of vector spaces, one for each vertex of the tree.
A problem suggested by Wolfgang Hackbusch and Joseph Landsberg is to compare the complexities of encodings, if one presents the same tensor with respect to two different trees.
We answer this question when the two trees are extremal cases:
the most ``spread'' tree (perfect binary tree), and
the ``deepest'' binary tree (train track tree).
The corresponding tensor formats are called hierarchical formats (HF) and tensor train (TT) formats, respectively.
\end{abstract}
\maketitle
\selectlanguage{french}
\begin{abstract}
Nous d\'emontrons une conjecture de W. Hackbusch concernant des r\'eseaux de tenseurs associ\'es \`a certains arbres. Les r\'eseaux de tenseurs sont utilis\'es pour pr\'esenter des tenseurs apparemment compliqu\'es d'une mani\` ere relativement simple et efficace.
Chaque pr\'esentation d'un tenseur donn\'e est d\'ecrite par un arbre binaire et une collection d'espaces vectoriels (un pour chaque sommet de l'arbre). Un probl\`eme pos\'e par Wolfgang Hackbusch et Joseph Landsberg est de comparer les complexit\'es des codages lorsque l'on pr\'esente le m\^eme tenseur via deux arbres diff\'erents. Nous r\'epondons \`a cette question lorsque les deux arbres sont des cas extr\^emes: l'arbre binaire parfait d'une part et l'arbre binaire le plus "profond" d'autre part. Les formats des tenseurs correspondants sont appel\'es format hi\'erarchique (HF) et format train de tenseurs (TT), respectivement.
\end{abstract}
\selectlanguage{british}
\maketitle

\section{Introduction}\label{sec_intro}

In many sciences tensors encode sophisticated data or algorithms. It is therefore desirable that the tensors are represented \emph{efficiently}.
It is very much dependent on the particular problem, what does ``efficiently'' mean.
One possibility is to study the rank of tensors, and express them as a sum of simple tensors,
   see \cite{landsberg_tensorbook, coppersmith_winograd_asymptotic,
    strassen_gaussian_elimination_is_not_optimal, landsberg_jabu_ranks_of_tensors, vrana2013asymptotic, OedingOttaviani}.
Another possibility, which we address in this article, is to use \emph{tensor formats},
  see also \cite{hackbusch_book, 75MR2844780, 152MR2573051, 155MR2837533, 159MR2566459, 191vidal2003efficient, 190verstraete2006matrix}.

A tensor format represents a tensor $t \in V_1 \otimes V_2\otimes \dotsb \otimes V_n$
   by a sequence of linear subspaces in two-fold tensor products.
The choice of a two-fold tensor products is determined by the combinatorics of a binary tree $\tree$.
The idea is that in many practical situations, the dimensions of the linear spaces involved are much smaller
   than the dimension of $V_1 \otimes V_2\otimes \dotsb \otimes V_n$.
Thus, in these situations, tensor formats provide an efficient method of encoding the tensors.

More precisely, for a binary tree $\tree$ with $n$ leaves we pick a vector space $V_i$ for each leaf.
For each vertex $v$ we pick an integer $f(v)$.
We define the \emph{variety of tensor network states} $\TNS_{\tree, f} \subset V_1 \otimes \dotsb \otimes V_n$ 
  in the following way:
  $t \in \TNS_{\tree, f}$ if and only if there exist linear subspaces $U_v$ of dimension at most $f(v)$,
  such that:
  \begin{itemize}
     \item $U_i \subset V_i$, if $v=i$ is one of the leaves,
     \item $U_v \subset U_{v_1}\otimes U_{v_2}$ whenever $v$ is not a leaf and $v_1$ and $v_2$ are its children,
     \item $t \in U_{\Root}$, if $v={\Root}$ is the root of the tree.
  \end{itemize}
\begin{figure}[htb]
\[
\xymatrix@C-2.2pc@R-1.2pc{ & & &                                     & \textroot \ar@{-}[lld] \ar@{-}[rrd]& & & & \\
           & & \node\ar@{-}[ld] \ar@{-}[rd] &&                       &             &\node\ar@{-}[ld] \ar@{-}[rd] & & \\
           & \node\ar@{-}[ld] \ar@{-}[d] &&\node\ar@{-}[ld] \ar@{-}[d]  &             &\node\ar@{-}[d] \ar@{-}[rd] & &\node\ar@{-}[d] \ar@{-}[rd] \\
\textleaf{}           &\textleaf{} &\textleaf{} &\textleaf{}&                                 &\textleaf{} &\textleaf{} & \textleaf{} & \textleaf{}\\
}
 \]
 \caption{Perfect binary tree of level $3$ (with $8 =2^3$ leaves).}\label{fig_perfect_tree}
\end{figure}
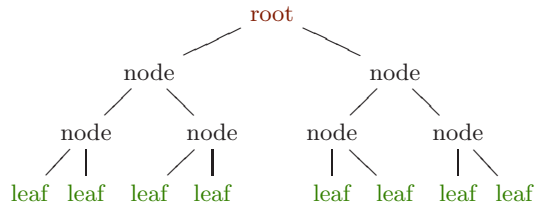
\begin{figure}[htb]
\[
\xymatrix@C-1pc@R-1.8pc{  &&&&&&                                   & \textroot \ar@{-}[ld] \ar@{-}[rd]&  \\
           &&&&&&                                 \node\ar@{-}[ld] \ar@{-}[rd] &&     \textleaf{} \\
           &&&&&                                 \node\ar@{-}[ld] \ar@{-}[rd] &&     \textleaf{} \\
           &&&&                                 \node\ar@{-}[ld] \ar@{-}[rd] &&     \textleaf{} \\
           &&&                                 \node\ar@{-}[ld] \ar@{-}[rd] &&     \textleaf{} \\
           &&                                 \node\ar@{-}[ld] \ar@{-}[rd] &&     \textleaf{} \\
           &                                 \node\ar@{-}[ld] \ar@{-}[rd] &&     \textleaf{} \\
                                            \textleaf{} &&     \textleaf{} \\
}
 \]
\caption{Train track tree  with 8 leaves.}\label{fig_train_tree}
\end{figure}
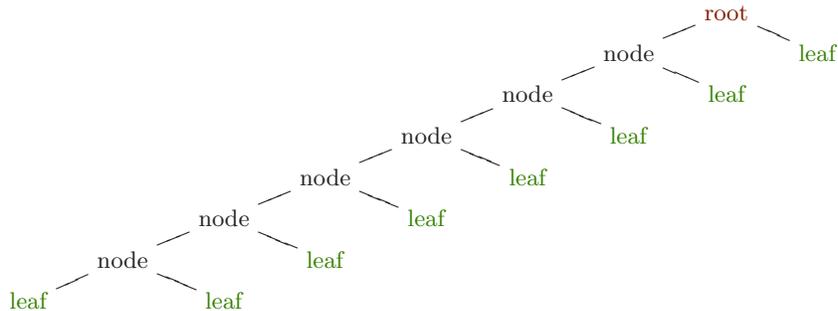
   For an elementary example see, for instance, Figure 11.1 and Equation (11.1) in \cite{hackbusch_book}.
In other words, in order to represent $t\in \TNS_{\tree, f}$ in the tensor format corresponding to $\tree$,
  we pick a linear subspace in each of the $V_i$, and then a sequence of subspaces of the tensor product
  $U_v \subset U_{v_1}\otimes U_{v_2}$, where $v$ is the parent of ${v_1}$ and ${v_2}$.
  Finally, we pick $t \in U_{\Root}$, if $v = \Root$ is the root.
In Section~\ref{sec_definitions} we present an equivalent definition not involving explicitly the subspaces $U_v$.

In this article we compare tensor formats for two different trees: \emph{perfect binary tree} $\HFT$ of level $k$ with $2^k$ leaves (Figure~\ref{fig_perfect_tree}),
  and \emph{train track tree} $\TTT$ (Figure~\ref{fig_train_tree}) with $n$ leaves.
In the literature,
    the corresponding tensor formats are called
    \emph{hierarchical tensor representation} \cite[Chapter~11]{hackbusch_book}
    and \emph{\TTformat} \cite{155MR2837533}, \cite[Chapter~12]{hackbusch_book}.
For simplicity of exposition most of the time we suppose that all integers $f(v)$ are equal to a fixed integer $r$
    and moreover that $\dim V_i \ge r$.
The variety of tensors of hierarchical format for a perfect binary tree $\HFT$ of depth $k$ (i.e.~with $n= 2^k$ leaves)
   is denoted $\HF(r,k):=\TNS_{\HFT, r}$.
The variety of tensors of \TTformat{} for a tensor train tree $\TTT$ with $n$ leaves is denoted $\TT(r,n):=\TNS_{\TTT, r}$.
We answer the following question, communicated to us by Joseph Landsberg,
   and motivated by a conjecture of Wolfgang Hackbusch \cite[Conj.~12.7]{hackbusch_book}.
\begin{question} Given integers $r$ and $k$, suppose $\dim V_i \ge r$ for each $i$.
  \begin{enumerate}
    \item What is the maximal integer $r'$ such that $\HF(r,k)$ is not contained in any variety $\TT(r',2^k)$
             for any choice of ordering of the leaves of the tensor train tree?
    \item What is the maximal $r'$ such that $\TT(r,2^k)$ is not contained in any variety $\HF(r',k)$
             for any choice of ordering of the leaves of the perfect binary tree?
\end{enumerate}
\end{question}

In particular, we prove the conjecture of Hackbusch.
\begin{thm}[Conjecture of Hackbusch]\label{thm_conj_hackbusch_intro}
   The variety $\HF(r,k)$ is not contained in $\TT(r',2^k)$ for $r'<r^{\lceil\frac{k}{2}\rceil}$
       for any ordering of leaves of $\TTT$.
\end{thm}
The bound on $r'$ in the theorem is optimal.

\begin{obs}\label{prop_easy_containment_HF_TT}
   The variety $\HF(r,k)$ is contained in $\TT(r^{\lceil\frac{k}{2}\rceil},2^k)$
       for the standard (left to right) ordering of leaves of both trees $\TTT$ and $\HFT$.
\end{obs}

Moreover, we also prove an analogous statement for the the other containment.

\begin{obs}\label{prop_containment_TT_HF}
   The variety $\TT(r,2^k)$ is contained in $\HF(r^2,k)$
       for the standard (left to right) ordering of leaves of both trees $\TTT$ and $\HFT$.
   However, $\TT(r,2^k)$ is not contained in $\HF(r',k)$ for any $r'<r^2$, for any reordering of the leaves of $\TTT$
\end{obs}

Proposition~\ref{prop_easy_containment_HF_TT} and the first part of Proposition~\ref{prop_containment_TT_HF}
   are proved in Section~\ref{sect_containments}.
Theorem~\ref{thm_conj_hackbusch_intro} and the second part of Proposition~\ref{prop_containment_TT_HF}
   are proved in Section~\ref{sect_hackbusch}.

\subsection*{Acknowledgements}

The authors are grateful to the organisers and lecturers of the summer school
``Interdisciplinary approach to tensor decomposition'' at Foundazione Bruno Kessler in Povo, Trento (14-18 July 2014).
The article is a result of the problem communicated there by Joseph Landsberg,
    and solved in the stimulative atmosphere of the school.
We thank Yang Qi for discussions and for explaining us the dimension formula (Proposition~\ref{prop_dimension_formula}).
The article is written as a part of "Computational complexity, generalised Waring type problems and tensor decompositions",
a project within "Canaletto",
    the executive program for scientific and technological cooperation between Italy and Poland, 2013-2015.
The paper is also a part of the activities of AGATES research group.

\section{Definitions}\label{sec_definitions}
Fix vector spaces $V_1,\dots, V_n$ and a tree $\Tr$ with a root $\Root$ and exactly $n$ leaves.
Throughout we assume that all trees are connected, and they are \emph{full binary} trees, that is each vertex
   either has exactly two children (then it is called a \emph{node})
   or it has no children at all (hence it is called a \emph{leaf}).
Two main trees that we are interested in are the following.

\dfi[The tree for hierarchical format, $\HFT$]
The perfect binary tree $\HFT$ of depth $k$ is a tree with $2^{k+1}-1$ vertices of which $2^{k}$ are leaves,
    such that every leaf has the same number of ancestors (equal to $k$).
Case $k=3$ is illustrated on Figure~\ref{fig_perfect_tree}.
This tree leads to hierarchical format of tensors, and its variety is denoted $\HF(r,k):=\TNS_{\HFT, r}$.
\kdfi

\dfi[The tree for \TTformat, $\TTT$]
   The train track tree $\TTT$ of $n$ leaves is a binary tree with $2n-1$ vertices,
      such that each node has at least one leaf as a child.
   Case $n=8$ is illustrated on Figure~\ref{fig_train_tree}.
   This tree leads to \TTformat{} of tensors, and its variety is denoted $\TT(r,n):=\TNS_{\TTT, r}$.
\kdfi

\dfi[Vertices $\V$, nodes $\NNN$, leaves $\LL$, $\cos$]
Let $\V$, $\NNN$, $\LL$ be respectively the set of vertices, nodes and leaves of the tree $\Tr$.
We have $\V = \NNN \sqcup \LL$.
For any vertex $v\in \V$ we denote by $\cos v$ the subset of leaves of $\Tr$ that are descendants of $v$.
\kdfi

The main object of our study is the variety of tensor network states
\[
  \TNS_{\Tr, r} \subset V_1\otimes\dots\otimes V_n.
\]
In order to describe it let us discuss tensor contractions.
\dfi[Contraction $\hook$]
We define the contraction map:
$$\hook :W_1^*\otimes (W_1\otimes W_2)\rightarrow W_2.$$
For $g\in W_1^*$ and $w_1\otimes w_2\in W_1\otimes W_2$ we denote the image by $g\hook(w_1\otimes w_2)$
   which is defined as $g(w_1)w_2$.
We extend this map by linearity to whole $W_1\otimes W_2$.
\kdfi

We commence with recalling the definition of $\TNS_{\Tr, f}$ from Introduction.

\begin{df}\label{def_TNS}
Let $f:\V\rightarrow\n$ be any function. Fix a tree $\Tr$ and pick an order of the leaves $\LL$.
The variety $\TNS_{\Tr, f}$ of \emph{tensor network states} associated to the tree $\Tr$ and the function~$f$
  is the set of tensors $t\in V_1\otimes\dots\otimes V_n$,
  such that for each vertex $v$ of $\Tr$ there exists a linear subspace $U_v$ of dimension at most $f(v)$, and:
  \begin{itemize}
     \item $U_v \subset V_i$, if $v\in \LL$ is the $i$-th leaf,
     \item $U_v \subset U_{v_1}\otimes U_{v_2}$ whenever $v$ is not a leaf and $v_1$ and $v_2$ are its children,
     \item $t \in U_{\mathfrak{r}}$, if $v=\Root$ is the root of the tree.
  \end{itemize}
\end{df}

We underline that the variety of tensor network states and its embedding
   also depends on the choice of order of leaves, which is implicit in the notation  $\TNS_{\Tr, f}$.
We can eliminate the vector spaces $U_v$ from the definition by replacing them with rank conditions on contractions.
\begin{obs}\label{prop_equivalent_def_of_TNS}
   Let $f$, $\Tr$ and the order of leaves be as in Definition~\ref{def_TNS}.
   The variety $\TNS_{\Tr, f}$ is the locus of tensors $t\in V_1\otimes\dots\otimes V_n$,
      such that for any vertex $v\in\V$ we have:
   \[
     \dim \left(\left(\bigotimes_{l\in\cos v}V_l\right)^*\hook t\right)\leq f(v).
   \]
\end{obs}
\begin{proof}
  Suppose $t \in\TNS_{\Tr, f}$ and pick the linear spaces $U_v$ as in Definition~\ref{def_TNS}.
  Then $t \in U_v \otimes \left(\bigotimes_{l\not\in\cos v}V_l\right)$,
  hence
  \[
    \dim \left(\left(\bigotimes_{l\in\cos v}V_l\right)^*\hook t \right) \le \dim U_v \le f(v).
  \]

  If $t\in V_1\otimes\dots\otimes V_n$ is such that
     $\dim \left(\left(\bigotimes_{l\in\cos v}V_l\right)^*\hook t\right)\leq f(v)$,
     then as the linear spaces $U_v$ we may take the images of contractions
     \[
        U_v := \left(\bigotimes_{l\not\in\cos v}V_l\right)^*\hook t
              \subset \bigotimes_{l\in\cos v}V_l.
     \]
\end{proof}
Thus set theoretically $\TNS_{\Tr, f}$ is defined by flattenings corresponding to vertices.
More precisely, for each $v\in V$ these are the $(f(v)+1)$-minors of the matrix with
  coefficients linearly depending on
  $t \in V_1\otimes \dotsb \otimes V_n$
  representing the contraction map:
  \[
     \bigotimes_{l \in \cos v}V_l^*\stackrel{\hook t}{\rightarrow} \bigotimes_{l\not\in\cos v} V_l.
  \]
We will frequently use the fact that the rank of this map is equal to  the rank of the dual map:
$$\bigotimes_{l\not\in\cos v} V_l^*\rightarrow\bigotimes_{l\in \cos v}V_l.$$


We have discussed the following dimension formula with Yang Qi, who has obtained it independently.
For $v\in\V$ denote by $v_1$ and $v_2$ its children.
To avoid redundant restrictions we always assume that:
\begin{itemize}
 \item for every $i \in \setfromto{1}{n}$ and for the $i$-th leaf $l\in \LL$ we have $f(l)\leq \dim V_i$, and
 \item $f(v)\leq f(v_1)f(v_2)$ for all $v\in \NNN$.
\end{itemize}
Otherwise, if one of these restrictions fails, say $f(v) > f(v_1)f(v_2)$,
   then the variety of tensor network states is equal to one with $f(v)$ replaced with $f(v_1)f(v_2)$.
Moreover, if for any vertex $v$ we have $f(v)=0$, then $\TNS_{\Tr, f}= \set{0}$.
Thus we suppose:
\begin{itemize}
 \item $f(v) \ne 0$ for all vertices $v$.
\end{itemize}

\begin{obs}[also obtained independently by Yang Qi]
\label{prop_dimension_formula}
With the assumption above, suppose $\dim V_i=d_i$.
Define the function $f'\colon \Vertices \to \n$ inductively:
   for the root $\Root$ let $f'(\Root) :=1$.
Then
\begin{itemize}
\item $f'(v_1):=\min(f'(v)f(v_2),f(v_1))$ and
\item $f'(v_2):=\min(f'(v)f(v_1),f(v_2))$.
\end{itemize}
Then the dimension of the variety $\TNS_{\tree,f}$ equals:
\[
   1+\sum_{v\in \NNN}\bigl( f'(v_1)f'(v_2)-f'(v)\bigr) f'(v)+\sum_{l_i\in\LL}\bigl(d_i-f'(l_i)\bigr) f'(l_i).
\]
If $f$ is constant and equal to $r$ 
then the dimension of $\TNS_{\tree,f}$ is:
$$nr^2(r-1)+r^2+\sum_{l_i\in \LL} (d_i-r)r.$$
Moreover $\TNS_{\tree,f}$ is an irreducible algebraic variety.
\end{obs}
\dow
First let us prove that $\dim \TNS_{\tree,f}$ is at least of dimension described above. Note that $f'(v)\leq f(v),f'(v_1)f'(v_2)$ for any vertex $v$.

Choose inductively, starting from leaves, general subspaces $U_{l_i}\subset V_i$, $U_v\subset U_{v_1}\otimes U_{v_2}$ of dimension $f'(v)$ for all vertices $v$. For each $v\in \NNN$ the possible subspaces $U_v$ are parametrized by a Grassmannian $Gr(f'(v),f'(v_1)f'(v_2))$ of dimension $(f'(v_1)f'(v_2)-f'(v))f'(v)$. Analogously, for $l\in \LL$ we obtain the Grassmannian $Gr(f'(l_i),d_i)$.
Additionally, choose a general tensor $t\in U_{\Root_1}\otimes U_{\Root_2}$. In order to prove the inequality, notice that $t$ belongs to $\TNS_{\tree,f}$. Moreover, we claim that for any $v\in \NNN$:
$$U_v=(\bigotimes_{l\not\in \cos v} V_l)^*\hook t.$$

This is true for the root. By induction, we have to show that $(U_{v_1})^*\hook U_v=U_{v_2}$. If $\dim U_{v_2}\leq \dim U_{v_1}$ the statement follows from the fact that $U_v$ contains a generic vector. Otherwise, $f'(v_1)=f(v_1)$, hence $\dim U_{v_2}\leq \dim U_{v_1}\dim U_v$. The statement follows, as the space $U_v$ was spanned by $\dim U_v$ general vectors of $U_{v_1}\otimes U_{v_2}$. 
Hence, different choices of spaces provide different points $t$ and the inequality
\[
   \dim \TNS_{\tree,f} \ge \underbrace{1}_{\dim U_{\Root}}
   +\sum_{v\in \NNN}\underbrace{\bigl(f'(v_1)f'(v_2)-f'(v)\bigr)f'(v)}_{\dim Gr(f'(v), U_{v_1}\otimes U_{v_2})}
   +\sum_{l_i\in\LL}\underbrace{\bigl(d_i-f'(l_i)\bigr)f'(l_i)}_{\dim Gr(f'(l_i), V_i)}
\]
follows.

To prove the other inequality, we claim that for any $t \in \TNS_{\tree,f}$
\[
  \dim(\bigotimes_{l\not\in \cos v} V_l)^*\hook t \le f'(v).
\]
The claim concludes the proof, as it shows that any $t$ is determined by the choices of subspaces $U_v$ of dimension $f'(v)$ and an element of $U_{\Root}$.

To show the claim, set $U_v:=(\bigotimes_{l\not\in \cos v} V_l)^*\hook t$.
We prove inductively that $\dim U_v\leq f'(v)$. For the root the statement is obvious. Observe that $\dim U_{v_1}\leq \dim U_v\dim U_{v_2}$, hence $\dim U_{v_1}\leq f'(v)f(v_2)$, and analogously for $U_{v_2}$.
This finishes the proof of the claim.
\kdow

For a tree $\Tr$, we say that a subset of its leaves $S \subset \Leaves$ is \emph{\DOAD{\Tr}
  (descendant or antidescendant for $\Tr$)}
  if there exists a vertex $v\in \Vertices$,
  such that $S$ is equal to $\cos v$ or the complement $\Leaves \setminus \cos v$.

\begin{lema}\label{lem_DOAD_set_bounds_dim_of_contraction}
  Fix a tree $\Tr$ and a \DOAD{\Tr} subset $S \subset \Leaves$.
  If $t\in \TNS_{\Tr, r}$ for an integer constant $r$, then
     $\dim ((\bigotimes_{l\in S} V_l^*)\hook t)\leq r$
\end{lema}
\begin{proof}
  Since $\dim ((\bigotimes_{l\in \cos v} V_l^*)\hook t) = \dim ((\bigotimes_{l\not\in \cos v} V_l^*)\hook t)$,
  the claim follows from Proposition~\ref{prop_equivalent_def_of_TNS}.
\end{proof}

\section{Containments of varieties of tensor network states}\label{sect_containments}

Again we fix vector spaces $V_1, \dotsc, V_n$.
In this section we want to compare the varieties of tensor network states for two different trees.
We want them both to be contained in the same tensor product $V_1 \otimes \dotsb \otimes V_n$,
   and hence we need to match the ordering of the leaves.
From now on, we will constantly assume, that the functions $f$
   giving the rank conditions are all constant,
   and that the dimensions $\dim V_i$ are at least the value of $f$.

We commence by proving certain containments of varieties of tensor network states.

\lem\label{lem:1}
Fix two binary trees $\Tr$ and $\Tr'$ of any shape with leaves labelled by the same set and
   pick a positive integer $c \in \N$.
Suppose for any vertex $v' \in \Vertices'$,
   either $\cos v'$ or the complement $\Leaves \setminus \cos v'$
   is a union of at most $c$ sets $S_i$, where each $S_i$ is \DOAD{\Tr}.
Then $\TNS_{\Tr, r} \subset \TNS_{\Tr', r^{c}}$.
\klem
\dow
Let us fix a tensor $t \in \TNS_{\Tr, r}$.
Consider any vertex $v'$ of $\Tr'$.
We have to prove that:
\[
  \dim (\bigotimes_{l\in\cos v'} V_l^*)\hook t\leq r^{c} \text{ or, equivalently, }
  \dim (\bigotimes_{l\not\in\cos v'} V_l^*)\hook t\leq r^{c}
\]
One of the sets $\cos v'$ or $ \Leaves \setminus \cos v'$ is a union of $c$ \DOAD{\Tr} sets $\fromto{S_1}{S_{c}}$.
Thus it is enough to show:
\[
  \dim (\bigotimes_{i=1}^{c} (\bigotimes_{l\in S_i} V_l^*))\hook t\leq r^{c},
\]
which is guaranteed by Lemma~\ref{lem_DOAD_set_bounds_dim_of_contraction},
  since $\dim (\bigotimes_{l\in S_i} V_l^*)\hook t\leq r$ for each $i$.
\kdow

We need a simple lemma about binary representations of numbers.
For any non-negative integer $i$, we define the function $\alpha(i)$ to be the number of digits equal to $1$
   in the binary representation of $i$.

\begin{lema}\label{lem_number_of_1_in_binary_expansions}
  Suppose $k$ and $j$ are non-negative integers satisfying $ j \le 2^k$.
  Then either $\alpha(j)$ or $\alpha(2^k-j)$ is at most $\lceil \frac{k}{2}\rceil$.
\end{lema}
\begin{proof}
  If $j = 2^k$, then the claim is true.
  Otherwise, for $i \in \setfromto{0,1}{2^k-1}$,
     the function $\alpha$ satisfies the following properties:
  \begin{itemize}
     \item $\alpha(i) + \alpha(2^k-1-i) = k$,
     \item $\alpha(i+1) \le \alpha(i)+1$.
  \end{itemize}
  By the first property, either $\alpha(j) \le \frac{k}{2}$ or $\alpha(2^k-j-1) < \frac{k}{2}$.
  In the first case we are done. In the second case, we apply the other property to obtain $\alpha(2^k-j) < \frac{k}{2}+1$,
  which implies the claim.
\end{proof}

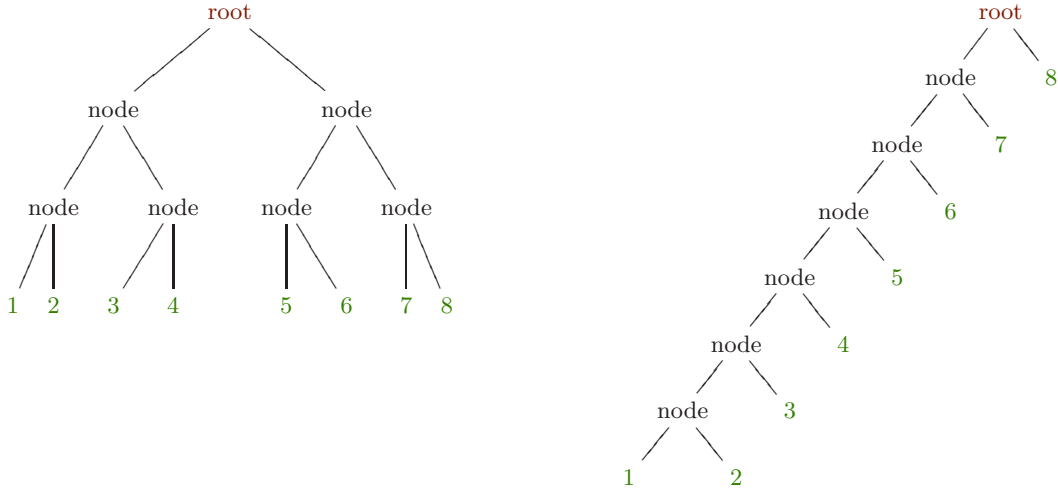
\begin{figure}[htb]
\begin{tabular}{cc}
\begin{minipage}[t]{0.47\textwidth}
\[
\xymatrix@C-2.3pc{ & & &                                     & \textroot \ar@{-}[lld] \ar@{-}[rrd]& & & & \\
           & & \node\ar@{-}[ld] \ar@{-}[rd] &&                       &             &\node\ar@{-}[ld] \ar@{-}[rd] & & \\
           & \node\ar@{-}[ld] \ar@{-}[d] &&\node\ar@{-}[ld] \ar@{-}[d]  &             &\node\ar@{-}[d] \ar@{-}[rd] & &\node\ar@{-}[d] \ar@{-}[rd] \\
\smalltextgreen{1}            &\smalltextgreen{2}&\smalltextgreen{3}&\smalltextgreen{4}&                                 &\smalltextgreen{5}&\smalltextgreen{6}& \smalltextgreen{7}& \smalltextgreen{8}\\
}
 \]
\end{minipage}
&
\begin{minipage}[t]{0.47\textwidth}
 \[
\xymatrix@C-2.5pc@R-1pc{  &&&&&&                                   & \textroot \ar@{-}[ld] \ar@{-}[rd]&  \\
           &&&&&&                                 \node\ar@{-}[ld] \ar@{-}[rd] &&    \smalltextgreen{8}\\
           &&&&&                                 \node\ar@{-}[ld] \ar@{-}[rd] &&
           \smalltextgreen{7} &\text{\phantom{\node}}\\
           &&&&                                 \node\ar@{-}[ld] \ar@{-}[rd] &&     \smalltextgreen{6}\\
           &&&                                 \node\ar@{-}[ld] \ar@{-}[rd] &&     \smalltextgreen{5}\\
           &&                                 \node\ar@{-}[ld] \ar@{-}[rd] &&     \smalltextgreen{4}\\
           \text{\phantom{\node}}&           \node\ar@{-}[ld] \ar@{-}[rd] &&     \smalltextgreen{3}\\
                                            \smalltextgreen{1}&&     \smalltextgreen{2}\\
}
 \]
\end{minipage}
\end{tabular}
 \caption{The standard (left to right) ordering of leaves on the perfect binary tree $\HFT$
          and the train track tree $\TTT$.}\label{fig_standard_order_of_leaves}
\end{figure}

\wn\label{lem:HFinTT}
For any $k,r\in\n$ the following inclusion holds:
$$\HF(r,k)\subset \TT(r^{\lceil \frac{k}{2}\rceil},2^k),$$
where the leaves in both trees are ordered from left to right, as shown in Figure~\ref{fig_standard_order_of_leaves}.
\kwn
\dow
Let $0 \le j\le 2^k$ be an integer.
Consider the binary representation of the number $j$,
   and the number $\alpha(j)$ of digits $1$ in the binary representation of $j$.
Then the initial segment  $[\fromto{1}{j}]$ is a union of $\alpha(j)$ sets of the form $\cos v$ for $v\in \HFT$.
Analogously, the complement $[\fromto{j+1}{2^k}]$ is a union of $\alpha(2^k-j)$ sets of such form.

By Lemma \ref{lem:1} it is enough to prove that each initial segment $[1,\dots,j]$
   or its complement $[\fromto{j+1}{2^k}]$ can be represented as descendants of at most $\lceil \frac{k}{2}\rceil$
   vertices of $\HFT$. This follows by Lemma~\ref{lem_number_of_1_in_binary_expansions}.
\kdow

\wn\label{cor_TT_in_HF}
For any $k,r\in\n$ the following inclusion holds:
$$\TT(r,2^k)\subset \HF(r^2,k),$$
where the leaves in both trees are ordered from left to right, as shown in Figure~\ref{fig_standard_order_of_leaves}.
\kwn
\dow
Let us fix a vertex $v$ of $\HFT$. Notice that $\cos v$ is an interval, say $[a,b]$. Its complement is the union of $[1,a-1]$ and the complement of $[1,b]$, i.e.~it is a union of two \DOAD{\TTT} sets. The conclusion follows by Lemma \ref{lem:1}.
\kdow

\section{Hackbusch Conjecture}\label{sect_hackbusch}

Our aim is to prove that the bounds in Corollaries \ref{lem:HFinTT} and \ref{cor_TT_in_HF} are optimal,
  even if we allow arbitrary ordering of leaves.
For the purpose of induction argument, we need Lemma~\ref{lem:gl_ogolny}.
The idea is that from tensors, which have large contractions on two disjoint subtrees
  we can construct tensors with large contractions on the whole tree.
So for tree $\Tr$ and two vertices $\Root', \Root'' \in \Vertices$,
denote by $\Tr'$ and $\Tr''$ the two subtrees of $\Tr$ with roots $\Root'$ and $\Root''$, and with leaves $\cos \Root'$ and
$\cos \Root''$, respectively,
  see Figure~\ref{fig_subtrees}.

In the proofs in this section, for the clarity and brevity of notation we will often write tensor products in a different order than originally.
For instance, in Equation~\eqref{equ_examples_of_order_of_tensor_product} below there are two disjoint sets of indices
  $\cos \Root', \cos \Root'' \subset \Leaves$, two tensors $t' \in \bigotimes_{l \in \cos \Root'} V_l$ and $t''\in \bigotimes_{l \in \cos \Root''} V_l$,
  and vectors $x_l \in V_l$.
To be formally correct, we should write:
\[
  t := t' \otimes t''\otimes \bigotimes_{l \not\in \cos \Root' \sqcup \cos \Root''} x_l \in
  \left(\bigotimes_{l \in \cos \Root'} V_l\right) \otimes \left(\bigotimes_{l \in \cos \Root''} V_l\right)
  \otimes \left(\bigotimes_{l \not\in \cos \Root' \sqcup \cos \Root''} V_l\right)
  \simeq V_1 \otimes \dotsb \otimes V_n.
\]
Instead, we skip the redundant middle term $\left(\bigotimes_{l \in \cos \Root'} V_l\right) \otimes \left(\bigotimes_{l \in \cos \Root''} V_l\right)
  \otimes \left(\bigotimes_{l \not\in \cos \Root' \sqcup \cos \Root''} V_l\right)$.

\begin{figure}[htb]
\[
\xymatrix@C-1.5pc{ && &\ar@{-}`ld[ldldld]`^r[ddddl]`^ru[dddrr]_{\text{\normalsize{$\Tr'$}}}`^u[rr]
`^lu[][]	&                                & \textroot \ar@{-}[lld] \ar@{-}[rrd]& &  &\text{\phantom{\node}}  & \\
           && & \text{\scriptsize{$\Root'$}} \ar@{-}[ld] \ar@{-}[rd] &&                       &             &\node\ar@{-}[ld] \ar@{-}[rd] &
           \ar@{-}`rd[dddr]`_l[ddd]`_ul[ddl]^{\text{\normalsize{$\Tr''$}}}`_u[l]`_ur[][]
           & \\
           && \node\ar@{-}[ld] \ar@{-}[d] &\text{\phantom{\node}} &\node\ar@{-}[ld] \ar@{-}[d]  &             &\node\ar@{-}[dd] \ar@{-}[rdd] & &\text{\scriptsize{$\Root''$}} \ar@{-}[d] \ar@{-}[rd] \\
&\textleaf{}&\textleaf{}&\textleaf{}&\textleaf{}&                                 &&&\textleaf{}&\textleaf{}\\
&&&&&&\textleaf{}&\textleaf{}&&&}
 \]
 \caption{Tree $\Tr$ with two subtrees $\Tr'$ and $\Tr''$ determined by vertices $\Root'$ and $\Root''$.}\label{fig_subtrees}
\end{figure}

\begin{lema} \label{lem:gl_ogolny}
   Fix any subset $A \subset \Leaves$ of leaves of a tree $\Tr$,
      and choose two disjoint subtrees $\Tr'$ and $\Tr''$ as above.
   Let $A' = A \cap \cos \Root'$ and $A'' = A \cap \cos \Root''$,
      and choose $t' \in \TNS_{\Tr', r}$ and $t'' \in \TNS_{\Tr'', r}$.
   Set $q' : = \dim \left(\left(\bigotimes_{l\in A'} V_l\right)^* \hook t'\right)$ and $q'' : = \dim \left(\left(\bigotimes_{l\in A''} V_l\right)^* \hook t''\right)$.
   Then there exists a tensor $t\in \TNS_{\Tr, r}$ such that
   \[
      \dim \left(\left(\bigotimes_{l\in A} V_l\right)^* \hook t\right) = q' q''.
   \]
\end{lema}


\dow
Pick any non-zero vectors $x_l \in V_l$ for each $l\not\in \cos \Root' \sqcup \cos \Root''$.
Then set
\begin{equation}\label{equ_examples_of_order_of_tensor_product}
  t := t' \otimes t''\otimes \bigotimes_{l \not\in \cos \Root' \sqcup \cos \Root''} x_l \in V_1 \otimes \dotsb \otimes V_n.
\end{equation}
We claim such $t$ satisfies the required properties.

To see that $t\in \TNS_{\Tr, r}$, pick any vertex $v \in \Vertices$;
       we have to show:
       \[
          \dim \left(\left(\bigotimes_{l\in \cos v} V_l\right)^* \hook t\right) \le r.
       \]
       If $v$ is in $\Tr'$, then
          the required dimension bound is provided by the condition $t' \in \TNS_{\Tr', r}$,
          since $t = t' \otimes x$, with $x \in \bigotimes_{l \not\in \cos v} V_l$.
       Similarly, if $v$ is in $\Tr''$.
       Otherwise, the contraction of $t$ is always $1$ dimensional.

To see $
      \dim \left(\left(\bigotimes_{l\in A} V_l\right)^* \hook t\right) = q' q''
       $,
we split the contraction into three parts:
\[
  \left(\bigotimes_{l\in A} V_l\right)^* \hook t =
  \underbrace{\left(\left(\bigotimes_{l\in A' } V_l\right)^* \hook t' \right)}_{\dim \le q' } \otimes
  \underbrace{\left(\left(\bigotimes_{l\in A''} V_l\right)^* \hook t''\right)}_{\dim \le q''} \otimes
  \underbrace{\left(\left(\bigotimes_{l\in A\setminus (A'\sqcup A'') } V_l\right)^*
                     \hook \bigotimes_{l\not\in \cos \Root' \sqcup \cos \Root''}x_l\right)}_{\dim =1}.
\]

\kdow

\begin{figure}[htb]
\[
\xymatrix@C-2.3pc{
& & & & & & & & & & \textroot \ar@{-}[lllld] \ar@{-}[rrrrrd]& & & &&&&&&&& \\
& & & \ar@{-}`dl[ddddlll]`^r[dddd]`^ru[dddrr]_{\text{\normalsize{$\Tr_1$}}}`^u[rr]`^lu[][]& & &\text{\scriptsize{$v_1$}}\ar@{-}[llld] \ar@{-}[rrd] & &\ar@{-}`dl^d[dddll]`_dl[ddddlll]`^dr[dddddl]`^ur[dddrr]`^u[ddrr]_{\text{\normalsize{$\Tr_2$}}}`^ul[][] & & & & &   \ar@{-}`dr_d[dddrr]`^dr[ddddrrr]`_dl[dddddr]`_ul[dddll]`_u[ddll]^{\text{\normalsize{$\Tr_3$}}}`_ur[][] & &\text{\scriptsize{$v_2$}}\ar@{-}[lld] \ar@{-}[rrrd] & & & \ar@{-}`dr[ddddrrr]`_l[dddd]`_lu[dddll]^{\text{\normalsize{$\Tr_4$}}}`_u[ll]`_ru[][]& && \\
& & &\text{\scriptsize{$\Root_1$}}\ar@{-}[ld] \ar@{-}[rd] & & & & &\text{\scriptsize{$\Root_2$}}\ar@{-}[ld] \ar@{-}[rd] & &&
                    & &\text{\scriptsize{$\Root_3$}}\ar@{-}[ld] \ar@{-}[rd] & & & & &\text{\scriptsize{$\Root_4$}}\ar@{-}[ld] \ar@{-}[rd] & & & \\
&&\node\ar@{-}[ld] \ar@{-}[d]&&\node\ar@{-}[ld] \ar@{-}[d]&&&\node\ar@{-}[ldd] \ar@{-}[dd]&&\node\ar@{-}[ld] \ar@{-}[d]&&&
  \node\ar@{-}[d] \ar@{-}[rd]&&\node\ar@{-}[dd] \ar@{-}[ddr]&&&\node\ar@{-}[d] \ar@{-}[rd]&&\node\ar@{-}[d] \ar@{-}[rd]&&\\
&\smalltextgreen{1}&\smalltextgreen{2}&\smalltextgreen{3}&\smalltextgreen{4}&&&&\smalltextgreen{7}&\smalltextgreen{8}&&&
 \smalltextgreen{9}&\smalltextgreen{10}&&&&\smalltextgreen{13}&\smalltextgreen{14}&\smalltextgreen{15}&\smalltextgreen{16}&\\
&&&&&& \smalltextgreen{5}&\smalltextgreen{6}&&&&&&& \smalltextgreen{11}&\smalltextgreen{12}&&&&&&\\
&&&&&&&&&&&&&&&&&&&&}
 \]
 \caption{Four subtrees $\Tr_1$, $\Tr_2$, $\Tr_3$, and $\Tr_4$ of $\HFT$.
          We denote vertices $\Root_1$, $\Root_2$, $\Root_3$, $\Root_4$, $v_1$, and $v_2$ as illustrated on the figure.}\label{fig_four_subtrees}
\end{figure}

In the proof of Conjecture of Hackbusch we will use the above lemma for the tree $\HFT$ and the corresponding variety of tensor network states $\HF(k,r)$.
The tree $\HFT$ contains four subtrees $\Tr_1,\Tr_2,\Tr_3,\Tr_4$,
   where the root of $\Tr_i$ is $\Root_i$
   and the leaves of $\Tr_i$ are $\cos \Root_i=\setfromto{2^{k-2}(i-1)+1}{2^{k-2}i}$, as illustrated in Figure~\ref{fig_four_subtrees}.
For each of the trees $\Tr_i$ we can also consider the variety of tensor network states $\TNS_{\Tr_i, r}$,
   with the underlying vector spaces $V_l$ for $l \in \setfromto{2^{k-2}(i-1)+1}{2^{k-2}i}$.

\lem\label{lem:gl}
Let us fix any subset $A$ of leaves of the tree $\HFT$. Suppose that:
\begin{enumerate}
\item for one of the four trees $\Tr_1,\Tr_2,\Tr_3,\Tr_4$ there exists tensor $t'' \in \TNS_{\Tr_i, r}$,
        such that
        \[
           \dim (\bigotimes_{l\in A\cap \cos \Root_i} V_l^*)\hook t''\geq r^{\lceil \frac{k}{2}\rceil-1},
        \]
\item there exists another subtree $\Tr'$ of $\HFT$ with a root $\Root'$ such that $\cos \Root' \cap\cos \Root_i=\emptyset$
       and $A \cap \cos {\Root'} \neq \emptyset$ and $\cos {\Root'} \nsubseteq A$,
          i.e.~the set of leaves of the tree $\Tr'$ and its complement have  non-empty intersection with the set $A$
          and it is disjoint from the set of leaves of $\Tr_i$.
       (If we set $\Tr'' = \Tr_i$, then this is illustrated on Figure~\ref{fig_subtrees}.)
\end{enumerate}
Then there exists a tensor $t\in \HF(k,r)$ such that:
$$\dim (\bigotimes_{l\in A} V_l^*)\hook t\geq r^{\lceil\frac{k}{2}\rceil}.$$
\klem
\dow
It is a direct application of Lemma~\ref{lem:gl_ogolny} with $\Tr''= \Tr_i$,
    and $q' \ge r$, $q'' \ge r^{\lceil \frac{k}{2}\rceil-1}$.
    The existence of $t''$ is given by the hypothesis of the lemma.
    As $t'$ we take a general tensor in $\TNS_{\Tr_i, r}$.
\kdow

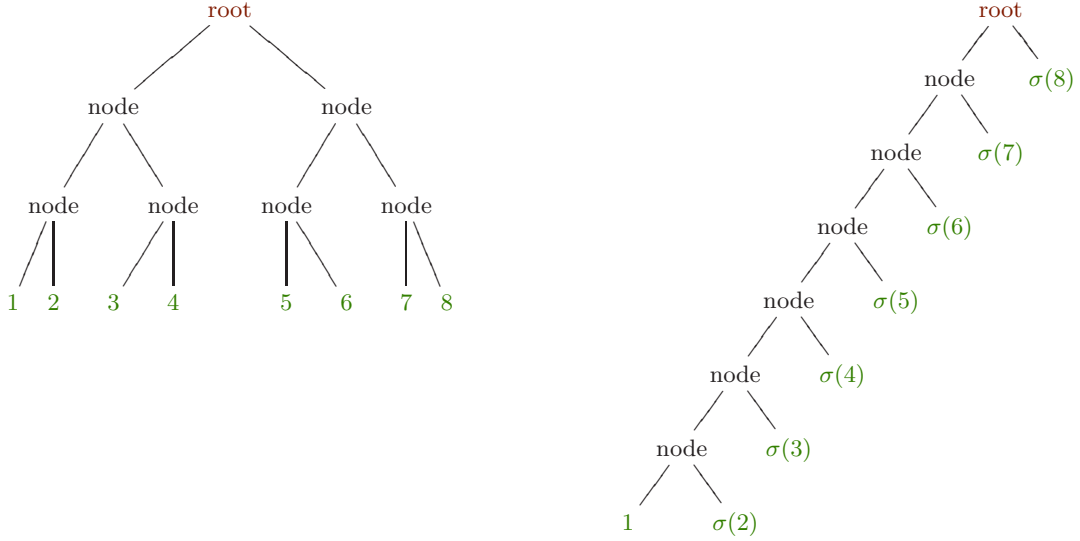
\begin{figure}[htb]
\begin{tabular}{cc}
\begin{minipage}[t]{0.47\textwidth}
\[
\xymatrix@C-2.3pc{ & & &                                     & \textroot \ar@{-}[lld] \ar@{-}[rrd]& & & & \\
           & & \node\ar@{-}[ld] \ar@{-}[rd] &&                       &             &\node\ar@{-}[ld] \ar@{-}[rd] & & \\
           & \node\ar@{-}[ld] \ar@{-}[d] &&\node\ar@{-}[ld] \ar@{-}[d]  &             &\node\ar@{-}[d] \ar@{-}[rd] & &\node\ar@{-}[d] \ar@{-}[rd] \\
\smalltextgreen{1}            &\smalltextgreen{2}&\smalltextgreen{3}&\smalltextgreen{4}&                                 &\smalltextgreen{5}&\smalltextgreen{6}& \smalltextgreen{7}& \smalltextgreen{8}\\
}
 \]
\end{minipage}
&
\begin{minipage}[t]{0.47\textwidth}
 \[
\xymatrix@C-2.5pc@R-1pc{  &&&&&&                                   & \textroot \ar@{-}[ld] \ar@{-}[rd]&  \\
           &&&&&&                                 \node\ar@{-}[ld] \ar@{-}[rd] &&    \smalltextgreen{$\sigma(8)$}\\
           &&&&&                                 \node\ar@{-}[ld] \ar@{-}[rd] &&
           \smalltextgreen{$\sigma(7)$} &\text{\phantom{\node}}\\
           &&&&                                 \node\ar@{-}[ld] \ar@{-}[rd] &&     \smalltextgreen{$\sigma(6)$}\\
           &&&                                 \node\ar@{-}[ld] \ar@{-}[rd] &&     \smalltextgreen{$\sigma(5)$}\\
           &&                                 \node\ar@{-}[ld] \ar@{-}[rd] &&     \smalltextgreen{$\sigma(4)$}\\
           \text{\phantom{\node}}&           \node\ar@{-}[ld] \ar@{-}[rd] &&     \smalltextgreen{$\sigma(3)$}\\
                                            \smalltextgreen{1}&&     \smalltextgreen{$\sigma(2)$}\\
}
 \]
\end{minipage}
\end{tabular}
 \caption{The standard ordering of leaves on the perfect binary tree $\HFT$
          and an arbitrary ordering of the leaves on the train track tree $\TTT$.
          The order is encoded by a permutation $\sigma \in S_{2^k}$.
          Without loss of generality we may assume that $\sigma(1) =1$.}\label{fig_arbitrary_order_of_leaves}
\end{figure}

We conclude with the proof of Hackbusch conjecture.

\tw[Hackbusch conjecture]
For any $k\geq 1$, $r \geq 2$ the variety $\HF(r,k)$ is not contained in any of the varieties $\TT(r^{\lceil \frac{k}{2}\rceil}-1,2^k)$ for any reordering of the leaves of $\TTT$.
\ktw
\dow
The ordering of leaves of the tree $\TTT$ will be encoded by a permutation $\sigma \in S_{2^k}$,
   where $S_{2^k}$ is the permutation group of $2^k$ elements, see Figure~\ref{fig_arbitrary_order_of_leaves}.
Without loss of generality, for the clarity of notation,  we may suppose that $\sigma(1)=1$.

The statement of the theorem is equivalent to the following:
for any permutation $\sigma\in S_{2^k}$ there exists a tensor $t\in \HF(r,k)$ and a number $1\leq j\leq 2^k$ such that
\begin{equation}\label{equ_ranks_to_prove_in_Hackbusch}
   \dim (\bigotimes_{i=1}^{j} V_{\sigma(i)}^*)\hook t\geq r^{\lceil \frac{k}{2}\rceil}.
\end{equation}

The proof is inductive on $k$.
If $k=0$, then there is nothing to do.
For $k=1$ the trees $\HFT$ and $\TTT$ are both equal to a binary tree with two leaves both connected to the root of the tree. The dimension bounds are $r$ for $\HF$ and $r-1$ for $\TT$, so  the statement is
$\HF(r,1)  \nsubseteq \TT(r-1,2)$
and amounts to an easy observation that a generic map has maximal rank.

Pick four subtrees $\Tr_1,\Tr_2,\Tr_3,\Tr_4$ of the tree $\HFT$ as in Figure~\ref{fig_four_subtrees}.
The permutation $\sigma$ induces an ordering $\fromto{1, \sigma(2)}{\sigma(n)}$ of the leaves of $\HFT$,
  hence also of leaves of each $\Tr_i$.

We consider the sequences of leaves $A_j:=\setfromto{\sigma(1)=1, \sigma(2)}{\sigma(j)}$ starting from $j=1$ and increasing $j$.
Some of the leaves corresponding to the elements of this sequence will belong $\Tr_1$, the others to $\Tr_2$ etc.
We keep increasing $j$, and we pause for a moment at the first instance $j=j_1$,
  when one of the four trees $\Tr_a$ has enough leaves to satisfy the induction assumption.
That is, set $j_1$ to be the smallest number such that for some $a \in \set{1,2,3,4}$ we have:
\[
  \dim (\bigotimes_{i\in A_{j_1} \cap \cos \Root_{a}} V_{i}^*)\hook t'' \geq r^{\lceil \frac{k}{2}\rceil-1}
\]
for some $t'' \in \TNS_{\Tr_a, r}$.
Using the inductive assumption, such $j_1$ exists.

Suppose there exists a leaf $l \in A_{j_1}$ which is outside of $\Tr_a$, say $l$ is a leaf of $\Tr_b$.
Note not all leaves of $\Tr_b$ are in $A_{j_1}$, as this would contradict minimality of $j_1$.
Thus, by Lemma~\ref{lem:gl}, the dimension bound in \eqref{equ_ranks_to_prove_in_Hackbusch} is satisfied.

Hence we may assume $A_{j_1}$ is contained in the set of leaves of $\Tr_a$. In particular, $\Tr_a = \Tr_1$, since $\sigma(1) =1$,
  and the first leaf is in $\Tr_a$.
Now we resume increasing $j$, until for some $j=j_2$ we obtain another subtree $\Tr_b \ne \Tr_1$,
  such that the leaves of $A_{j_2} \cap \cos \Root_b$ satisfy induction assumption:
\begin{equation}\label{equ_inductive_t_double_prime}
  \dim (\bigotimes_{i\in A_{j_2} \cap \cos \Root_{b}} V_{i}^*)\hook t'' \geq r^{\lceil \frac{k}{2}\rceil-1}
\end{equation}
for some $t'' \in \TNS_{\Tr_b, r}$.
Similarly as before, if $A_{j_2}$ contains a leaf from one of the subtrees other than $\Tr_1$ or $\Tr_b$,
  then by Lemma~\ref{lem:gl} the dimension bound \eqref{equ_ranks_to_prove_in_Hackbusch} is satisfied.
The same happens if not all leaves from $\Tr_1$ are contained in $A_{j_2}$.
Also, if $b\ne 2$, then as $\Tr'$ in Lemma~\ref{lem:gl} we may take the tree,
  whose root is $v_1\in \Vertices$, the son of $\Root$, and the father of $\Root_1$ and $\Root_2$.

Denote by $v_2$ the other son $\Root$, and the father of $\Root_3$ and $\Root_4$.

From now on we suppose that $b=2$ and $A= A_{j_2}$ contains all the leaves of $\Tr_1$, some of the leaves of $\Tr_2$,
   and no leaf of $\Tr_3$ or $\Tr_4$.
We will construct tensor $t$ satisfying \eqref{equ_ranks_to_prove_in_Hackbusch} for $A=A_{j_2}$ as $t:=u \otimes t''$,
  where $t'' \in \bigotimes_{i\in \cos \Root_{2}} V_{i}$ is the tensor from \eqref{equ_inductive_t_double_prime},
  and $u \in (\bigotimes_{i\in \cos \Root_{1}} V_{i}) \otimes (\bigotimes_{i\in \cos v_2} V_{i})$ is chosen as follows.
For each $v\in \Vertices$ pick a linear subspace $U_v$ of dimension $r$ as in Definition~\ref{def_TNS}.
Then let $u \in U_{\Root_1} \otimes U_{v_2}$ be general. Hence
\begin{align*}
  t  & \in        U_{\Root_1} \otimes \left(\bigotimes_{i\in \cos \Root_{2}} V_{i}\right) \otimes U_{v_2}\\
     & =    \left(U_{\Root_1} \otimes \bigotimes_{i\in A_{j_2} \cap \cos \Root_{2}} V_{i}\right)
    \otimes \left(U_{v_2} \otimes \bigotimes_{i\in \cos \Root_{2} \setminus A_{j_2} } V_{i} \right)
\end{align*}
To show \eqref{equ_ranks_to_prove_in_Hackbusch} for $A=A_{j_2}$, we must consider the dimension of the contraction:
\begin{align*}
  \left(U_{\Root_1} \otimes \bigotimes_{i\in A_{j_2} \cap \cos \Root_{2}} V_{i}\right)^* & \hook \underbrace{\left(u \otimes t''\right)}_{= t}  &=&&
  \underbrace{\left(U_{\Root_1} \hook u\right)}_{ = U_{v_2}} \otimes \underbrace{\left(\bigotimes_{i\in A_{j_2} \cap \cos \Root_{2}} V_{i}\right)^*  \hook t''}_{\text{has dimension at least } r^{\lceil \frac{k}{2}\rceil-1}}.
\end{align*}
Thus the dimension of the contraction is at least $r^{\lceil \frac{k}{2}\rceil}$, which completes the proof of the theorem.
\kdow

Naively, one could expect, that the dimensions of contractions are as large as possible.
More precisely, for any subset $A\subset \Leaves$ let $h(A)$ be defined as:
\[
   h(A):= \min \set{a : \exists \fromto{v_1}{v_a}\in \Vertices , \text{ such that } A = \bigcup_{i=1}^a \cos v_i}.
\]
Now take a general tensor $t\in \HF(r,k)$
    one could expect $\dim (\bigotimes_{l\in A}V_l^*\hook t)$ to be equal to $r^{\min(h(A),h(\Leaves \setminus A))}$.
The following example shows this is false.
\ex
Consider $k=3$ and $r=2$. For a generic tensor $t\in \HF(r,k)$
    we have $\dim (V_1^*\otimes V_2^*\otimes V_3^*\otimes V_5^*)\hook t=5<8$.
\kex
\uwa
It is interesting to see which contraction shows that $\HF(r,k)\not\subset \TT(r^{\lceil \frac{k}{2}\rceil}-1,2^k)$ with the same ordering of leaves. One can show that for $k$ odd it can be given by $\bigotimes_{i=1}^a V_i^*$ where the binary representation of $a$ is given by $1010\dots10+1$ where the number of digits equals $k-1$.
\kuwa
\ob
For any $k\geq 3$, $r \geq 2$ the variety $\TT(r,2^k)$
   is not contained in any of the varieties $\HF(r^2-1,k)$ for any matching of the leaves of the trees $\TTT$ and $\HFT$.
\kob
\dow
Let $n= 2^k$ be the number of leaves of $\TTT$ and $\HFT$.
Pick a matching of the leaves  determined by a permutation $\sigma$ as in Figure~\ref{fig_arbitrary_order_of_leaves}.
A \emph{cherry} in the tree $\HFT$ is a pair of leaves $(2i-1, 2i)$ for some $i \in \setfromto{1,2}{2^{k-1}}$.
Since $k\geq 3$, there exists a cherry $(2i-1, 2i)$ such that $1 < \sigma(2i-1), \sigma(2i) < n$.
That is, the leaves in $\TTT$ corresponding to the cherry are not among the extremal leaves $1$ or $n$.
Denote $a:= \sigma^{-1}(1)$ and $b:= \sigma^{-1}(n)$, i.e.~$a$ and $b$ are the leaves of $\HFT$
   corresponding to the extremal leaves of $\TTT$.

We claim that there exists a tensor $t\in \TT(r,2^k)$ such that $\dim \left(V_{2i-1}^*\otimes V_{2i}^*\hook t\right)=r^2$.
In particular, $t \notin \HF(r^2-1,k)$.
Explicitly, pick four $r$ dimensional subspaces
\[
   U_{a} \subset V_{b}, U_{a} \subset V_b, U_{2i-1} \subset V_{2i-1}, U_{2i} \subset V_{2i},
\]
 and vectors $x_l \in V_l$ for all $l\notin \set{a, b, 2i-1, 2i}$. To fix the notation, we assume that  \mbox{$\sigma(2i-1)<\sigma(2i)$} (otherwise, swap the roles of $\sigma(2i-1)$ and $\sigma(2i)$). Let $y_1\in U_a\otimes U_{\sigma(2i-1)}$ and $y_2\in U_{\sigma(2i)}\otimes U_b$ be general vectors.

We define
\[
       t := y_1\otimes y_2 \otimes \bigotimes_{l \notin \set{a, b, 2i-1, 2i}} x_l \in \bigotimes_{l \in \setfromto{1}{2^k} }V_l.
\]
The following calculation proves the claim:
\begin{align*}
  \dim \left(V_{2i-1}^*\otimes V_{2i}^*\hook t\right) &
    = \dim \left(\left(U_{2i-1}^*\hook y_1\right)\otimes \left(U_{2i}^*\hook y_2\right) \otimes \bigotimes_{l \notin \set{a, b, 2i-1, 2i}} x_l \right)\\
   &= \dim \left(U_{a}\otimes U_{b}\otimes \bigotimes_{l \notin \set{a, b, 2i-1, 2i}} x_l \right)\\
   &= \dim \left(U_{a}\otimes U_{b}\right) = r^2.
\end{align*}
It is a straightforward check that $t\in \TT(r,2^k)$.
\kdow

\bibliography{Hackbusch_Conj}
\bibliographystyle{alpha_four}

\end{document}